\documentclass[12pt]{article}
\usepackage{amsmath,amsthm,amsfonts,amssymb,verbatim}
\usepackage[usenames]{color} 
\newtheorem{thm}{Theorem}[section]
\newtheorem{lemma}[thm]{Lemma}
\newtheorem{cor}[thm]{Corollary}

\newtheorem{prop}[thm]{Proposition}

\newtheorem{sublemma}[thm]{Sublemma}

\theoremstyle{definition}  

\numberwithin{equation}{section}

\newtheorem{remark}[thm]{Remark}
\theoremstyle{definition}
\theoremstyle{remark}

\newcounter{enumitemp}

\DeclareMathOperator{\Fix}{Fix}
\DeclareMathOperator{\Aut}{Aut}

\DeclareMathOperator{\Diff}{Diff}

\DeclareMathOperator{\mcg}{MCG}

\DeclareMathOperator{\Homeo}{Homeo}

\newcommand{\cC}{{\cal C}}
\newcommand{\C}{{\mathbb C}}

\newcommand{\R}{\mathbb R}

\newcommand{\cP}{{\cal P}}

\newcommand{\T}{\mathbb T}

\newcommand{\Z}{\mathbb Z}

\def\D{{\mathbb D}}

\def\sl3z{SL(3, \mathbb Z)}

\newcommand{\trycomment}[1]{}

\title{Triviality of some representations of 
$\mcg(S_g)$ in $GL(n,\C), \Diff(S^2)$ and $\Homeo(\T^2)$}
\author{John Franks,\thanks{Supported in part by NSF grant DMS0099640.}\ \ 
Michael Handel\thanks{Supported in part by NSF grant DMS0103435.}}

\begin{document}
\maketitle
\begin{abstract}
We show the triviality of representations 
of the mapping class group of a genus $g$ surface in 
$GL(n,\C), \Diff(S^2)$ and $\Homeo(\T^2)$ when 
appropriate restrictions on the genus $g$ and the size of $n$
hold.  For example, if $S_g$ is a surface of finite type
and $\phi : \mcg(S_g) \to GL(n,\C)$ is a
homomorphism, then $\phi$ is trivial provided the genus 
$g \ge 3$ and $n < 2g.$  We also show
that if $S_g$ is a closed surface with genus $g \ge 7$,
then every homomorphism $\phi: \mcg(S_g) \to \Diff(S^2)$
is trivial and that if $g \ge 3$,
then every homomorphism  $\phi: \mcg(S_g) \to \Homeo(\T^2)$
is trivial.
\end{abstract}
\section{Introduction and Statement of Results} 

If  $S$ is a surface of genus $g$ with a (perhaps empty) finite set of
punctures and  boundary components we will denote by 
$\mcg(S)$ the group of isotopy classes of homeomorphisms
of $S$ which pointwise fix the boundary and punctures of $S$.
In case the boundary is non-empty the isotopies defining 
$\mcg(S)$ must pointwise fix the boundary.  This is a well
studied object (see the recent book by Farb and Margalit \cite{FM}
for an extensive survey of known results).  In this article
we show the triviality of representations 
of $\mcg(S)$  in 
$GL(n,\C), \Diff(S^2)$ and $\Homeo(\T^2)$ under various
additional hypotheses.

For a closed surface $S$ of genus $g$ there is a natural
representation of $\mcg(S)$ into the group of symplectic matrices of
size $n = 2g$ obtained by taking the induced action on
$H_1(S,\R).$ It is natural to ask if there are linear representations
of lower dimension. This is one of the questions we address.

\begin{thm}\label{thm:lin-rep}  Suppose that 
 $S$ is a genus $g\ge 1$ surface of finite type (perhaps with boundary
and punctures), that $n < 2g$  and that  $\phi : \mcg(S) \to GL(n,\C)$ is a
homomorphism.   Then the image $H$ of $\phi$ is abelian.  If $g = 2$ then $H$ is finite cyclic; if $g \ge 3$ then $H$ is trivial.  
\end{thm}

This result improves a result of L. Funar \cite{Funar} which
shows that homomorphisms from the genus $g \ge 1$ mapping class group
to $GL(n,\C)$ have finite image provided $n \le \sqrt{g+1}.$

It has come to our attention that our Theorem~(\ref{thm:lin-rep})
has also been proved independently by Mustafa Korkmaz \cite{kork3}.

Since the induced map on homology gives a non-trivial homomorphism from 
$\mcg(S)$ to $GL(2g,\C)$ for a closed surface $S$ of genus $g$, it is
clear that one cannot improve on the restriction that $n < 2g.$

In light of the result of Morita which asserts that there is
no lifting of $\mcg(S_g)$ to $\Diff(S_g)$, i.e.  no homomorphism $\phi
\mcg(S_g) \to \Diff(S_g)$, with $\phi(\sigma)$ a representative of the
isotopy class $\sigma$, it is natural to ask about the existence of
homomorphisms $\phi \mcg(S_g) \to \Diff(S_{g'})$.  An immediate
corollary of a recent result of Aramayona and Souto, \cite{AandS},
shows that in some ranges of values for $g$ and $g'$ any such
homomorphism is trivial.

More precisely they show:

\begin{thm}[Aramayona \& Souto \cite{AandS}]\label{AandS}
Suppose that $S$ and $S'$ are closed surfaces, where $S$ has 
genus $g \ge 6$ and $S'$ has genus less than $2g - 1$.  Then
any non-trivial homomorphism $\phi: \mcg(S) \to \mcg(S')$
is an isomorphism and in particular $S$ and $S'$ have the
same genus.
\end{thm}

A nearly immediate corollary is that representations of
$\mcg(S)$ into the group $\Homeo(S')$ of orientation
preserving homeomorphisms of a surface $S'$ with lower
genus are frequently trivial.

\begin{cor}
Suppose that $S$ and $S'$ are closed surfaces, where $S$ has 
genus $g \ge 6$ and $S'$ has genus $g',$ with $1 <g'< 2g - 1.$
Then every homomorphism $\phi: \mcg(S) \to \Homeo(S')$
is trivial.
\end{cor}

\begin{proof}
By Theorem~(\ref{AandS}) the image of the map to $\mcg(S')$ induced by
$\phi$ is trivial so the image of $\phi$ lies in $\Homeo_0(S')$, the
subgroup of homeomorphisms isotopic to the identity.  If $g \in
\mcg(S)$ has finite order then $\phi(g) \in \Homeo_0(S')$ has finite
order.  Since $S'$ has genus at least two and $\phi(g)$ is both finite
order and isotopic to the identity, it follows that $\phi(g) = id.$
Since $\mcg(S)$ is generated by elements of finite order (see Theorem
7.16 of \cite{FM}), we conclude that $\phi$ is trivial.
\end{proof}

This motivates the question of whether homomorphisms from $\mcg(S)$ to
$\Homeo(S')$ or $\Diff(S')$ are trivial when $S' = S^2$ or $\T^2.$

A classical result of Nielsen states that if $S$ has finite negative
Euler characteristic and is not closed then there are injective
homomorphisms from $\mcg(S)$ to $\Homeo(S^1)$ and hence by coning
such examples it follows that there are also injective
homomorphisms from $\mcg(S)$ to $\Homeo(\D^2)$ and $\Homeo(S^2)$.
Similarly taking the diagonal action we obtain an injective
homomorphisms from $\mcg(S)$ to $\Homeo(\T^2)$
We therefore assume, in the following two results, that our surfaces are
closed. 

\begin{thm}\label{thm:S2}
Suppose that $S$ is a closed surface with 
genus $g \ge 7$.
Then every homomorphism $\phi: \mcg(S) \to \Diff(S^2)$
is trivial.
\end{thm}

\begin{thm}\label{thm:torus}
Suppose that $S$ is a closed oriented surface with 
genus $g \ge 3$.
Then every homomorphism  $\phi: \mcg(S) \to \Homeo(\T^2)$
is trivial.
\end{thm}

In the case of $\Diff(S^2)$ we can deal with punctured
surfaces $S$ provided the number of punctures is bounded
by $2g -10$ where $g$ is the genus.

\begin{thm}\label{thm:punc}
Suppose that $S$ is an oriented surface with 
genus $g$ with $k$ punctures and no boundary.  
If $g \ge 7$ and $0 \le k \le 2g-10,$
then every homomorphism $\phi: \mcg(S) \to \Diff(S^2)$
is trivial.
\end{thm}

\begin{remark}
Note that by the coning construction described above, if
$S$ is a surface of genus $\ge 3$ with any number of punctures,
$\mcg(S)$ acts faithfully by {\em homeomorphisms} on $S^2$.  But
Theorem~(\ref{thm:punc}) asserts that if the genus $g$ is $\ge 7$ 
and the number of punctures is $\le 2g-10$ 
then any action by {\em diffeomorphisms} is trivial.
\end{remark}

The authors would like to thank Kiran Parkhe for pointing out
some inaccuracies in an earlier version of this paper.

\section{Criteria for Triviality}

In this section we recall or establish a few elementary properties
of $\mcg(S)$ which we will need. In particular 
Proposition~(\ref{prop:hom}) 
provides useful elementary criteria which imply that a homomorphism
from $\mcg(S)$ to a group $G$ is trivial.

\begin{prop}\label{prop:homology}
Let $S$ denote a surface of genus $g \ge 2$ which has finitely
many (perhaps $0$) punctures and/or boundary components.
Then the abelianization of $\mcg(S)$ is trivial if $g \ge 3$ and $\Z/10\Z$ if
$g = 2.$
\end{prop}

A proof of this result can be found as Theorem (5.1) of \cite{kork}.

\begin{lemma}\label{lem:solvable}
If $G$ is a perfect group and $H$ is solvable then
any homomorphism $\phi: G \to H$ is trivial.
\end{lemma}
\begin{proof}
Since $G$ is perfect any homomorphism from $G$ to an abelian
group is trivial.  Since $H$ is solvable we can inductively define $H_k$
by $H_0 = H$ and $H_{k} = [H_{k-1}, H_{k-1}]$ and then there exists
$n$ such that $H_n$ is the trivial group.  Since $H_k/H_{k+1}$
is abelian any homomorphism from $G$  to it is trivial.  Induction on
$k$  using the exact sequence $1 \to H_{k+1} \to H_k \to H_k/H_{k+1} \to 1$ shows that $\phi(G) \subset H_k$ for all $k$ and hence $\phi$
is trivial.
\end{proof}

The following  lemmas  are  well known; see Sections 1.3.2 -1.3.3 and 4.4.2-4.4.3 of   \cite{FM}.

\begin{lemma}  \label{change of coordinates} Let $S$ be an orientable surface of finite 
type (perhaps with boundary
and punctures) which has genus $g \ge 2.$  Then $\mcg(S)$ acts transitively on
\begin{enumerate} 
\item   the set of isotopy classes of non-separating simple closed curves in $S$.
\item  the set  $\cP$  of ordered pairs of disjoint simple closed curves in $S$ whose
union does not separate $S$.
\item  the set of ordered pairs of   simple closed curves in $S$ that intersect once.
\end{enumerate}
\end{lemma}

\begin{cor}\label{cor:pairs in P}
Suppose $S$ is as in the previous lemma and $\alpha$ and $\beta$
are simple closed curves in $S$.  If $\alpha$ and $\beta$
intersect transversely in a single point there is a simple
closed curve $\gamma$ such that both $(\alpha, \gamma)$
and $(\beta, \gamma)$ are in $\cP.$ Alternatively if $(\alpha, \beta) \in \cP$
then there exists a simple closed curve $\gamma$ which 
intersects both $\alpha$ and $\beta$ transversely in a single point.
\end{cor}

\begin{proof} It is straightforward to construct one example of
a pair $\alpha'$ and $\beta'$ which intersect in a single point
and a $\gamma'$ such that both $(\alpha', \gamma')$
and $(\beta', \gamma')$ are in $\cP.$  By Lemma~(\ref{change of coordinates})
there is a homeomorphism $h$ which takes $\alpha'$ to $\alpha$ and
$\beta'$ to $\beta$.  Letting $\gamma = h(\gamma')$ proves the first
assertion.

Similarly it is straightforward to construct one example of
a pair $(\alpha',\beta') \in \cP$ 
and a $\gamma'$ which intersects both $\alpha'$ and $\beta'$ transversely
in a single point. By Lemma~(\ref{change of coordinates})
there is a homeomorphism $g$ which takes $\alpha'$ to $\alpha$ and
$\beta'$ to $\beta$.  Letting $\gamma = g(\gamma')$ proves the 
second assertion.
\end{proof}

For an essential simple closed curve $\alpha$ in $S$, we denote the
{\em left Dehn twist about $\alpha$} by $T_\alpha.$ Note that by
Lemma~\ref{change of coordinates} (1), $T_\alpha$ is conjugate to
$T_\beta$ for all non-separating
simple closed curves $\alpha, \beta$ in $S$.

\begin{lemma} \label{generating set} Suppose $S$ is a surface
of finite type perhaps with boundary or punctures.  
There is a finite set $\cC_0$ of
non-separating simple closed curves in $S$ such that $\{T_\alpha:
\alpha \in \cC_0\}$ generates $\mcg(S)$ and such that for each
$\alpha, \beta \in \cC_0$, either 
\begin{enumerate}
\item  $(\alpha,\beta) \in \cP$, or
\item $\alpha$ and $\beta$ intersect transversely in a
single point, or
\item there is $\gamma \in \cC_0$ which intersects 
both $\alpha$ and $\beta$ transversely in a single point.
\end{enumerate}
\end{lemma}

For a proof of this result see the book of Farb and Margalit \cite{FM}
and, in particular, Figure 4.10 and the discussion after Corollary 4.16.
Note that the three possibilities are not mutually exclusive.
Indeed Corollary~(\ref{cor:pairs in P}) implies that if (1) holds there
is a simple closed curve $\gamma$ (not necessarily in $\cC_0$) 
which intersects both $\alpha$ and $\beta$ transversely in a single point.

\begin{prop}\label{prop:hom}  Suppose that 
$S$ is an orientable surface of finite type (perhaps with boundary
and punctures) with genus $g \ge 2$, and  
that $\phi: \mcg(S) \to G$ is a homomorphism to some group $G$.   
Suppose that one of the following holds:
\begin{enumerate}
\item There are disjoint simple closed curves $\alpha_0$ and $\beta_0$ 
in $S$ whose union does not separate (i.e. $(\alpha_0,\beta_0) \in \cP$)
 and $\phi(T_{\alpha_0}) = \phi(T_{\beta_0})^{\pm 1},$ or
\item  There are disjoint simple closed curves  $\gamma_0$ and $\delta_0$
in $S$ intersecting transversely in one
point such that $\phi(T_{\gamma_0})$ commutes with $\phi(T_{\delta_0}).$ 
\end{enumerate}
 Then the image of $\phi$ is a finite cyclic group.  If $g \ge 3$ then
$\phi$ is trivial.
\end{prop}

\begin{proof} We first prove (1). Let $\cC_0$ be
the set of simple closed curves described in 
Lemma~(\ref{generating set}), so Dehn twists
around elements of $\cC_0$ form a set of generators
of $\mcg(S).$ 
We want to show that Dehn twists around any two elements
of $\cC_0$ have the same image under $\phi$ or images which are
inverses in $G$.  We suppose 
$\alpha, \beta \in \cC_0$ and consider the three possibilities enumerated in 
Lemma~(\ref{generating set}).

If $(\alpha,\beta) \in \cP$ then by Lemma~\ref{change of
  coordinates} (2) there exists $\bar f \in \mcg(S)$ such that
$T_{\alpha} = \bar f T_{\alpha_0} \bar f^{-1}$ and $T_{\beta} = \bar f
T_{\beta_0} \bar f^{-1}$.  Thus
\begin{align*}
\phi(T_{\alpha}) &= \phi(\bar f T_{\alpha_0} \bar f^{-1})\\
&= \phi(\bar f) \phi( T_{\alpha_0}) \phi(\bar f^{-1})\\
&= \phi(\bar f) \phi( T_{\beta_0})^{\pm 1} \phi(\bar f^{-1})\\
&= \phi(\bar f T_{\beta_0} \bar f^{-1})^{\pm 1}\\
&= \phi(T_{\beta})^{\pm 1}.
\end{align*}
This proves that $\phi(T_\alpha) = \phi(T_\beta)^{\pm 1}$
when $(\alpha,\beta) \in \cP$.  
Note we used only the fact that $(\alpha,\beta) \in \cP$ 
and not the fact that $\alpha,\beta \in \cC_0.$

Next suppose that $\alpha$ and $\beta$ intersect in a single point.
By Corollary~(\ref{cor:pairs in P}) there exists a simple closed
curve $\delta$ (not necessarily in $\cC_0$)
such that $(\alpha, \delta), (\delta,\beta) \in \cP$.  Thus by
what we showed above, $\phi(T_\alpha)
= \phi(T_\delta)^{\pm 1}$ and $\phi(T_\beta) = \phi(T_\delta)^{\pm 1}$
so we have the desired result in
case (2) of  Lemma~\ref{generating set}.

Finally for case (3) of Lemma~\ref{generating set},
suppose there exists $\gamma \in \cC_0$ which intersects
both $\alpha$ and $\beta$ in a single point.  Then by what we have
just shown $\phi(T_\alpha) = \phi(T_\gamma)^{\pm 1}$ and
$\phi(T_\beta) = \phi(T_\gamma)^{\pm 1}.$

Hence in all three possibilities enumerated in Lemma~\ref{generating
set} we have $\phi(T_\alpha) = \phi(T_\beta)^{\pm 1}$.  Since $\cC_0$ is a
generating set, the image of $\phi$ is contained in a cyclic subgroup
of $G$. Hence $\phi$ factors through the abelianization
of $\mcg(S)$ and the result follows from Proposition~(\ref{prop:homology}).
This proves (1).

To prove (2) we observe that there are generators of $\mcg(S)$
consisting of Dehn twists about non-separating simple closed curves
each pair of which are either disjoint or intersect transversely in
a single point (see Corollary 4.16 and ff. of \cite{FM}).  It
suffices to show that the $\phi$-images of these generators commute,
for then $\phi(\mcg(S))$ is abelian 
and $\phi$ factors through the abelianization
of $\mcg(S)$ so the result follows from Proposition~(\ref{prop:homology}).

It is obvious that Dehn twists
about disjoint curves commute, so it suffices to show that if $\gamma$ and
$\delta$ intersect transversely in a single point then
$\phi(T_\gamma)$ and $\phi(T_\delta)$ commute.  By hypothesis this is
true for one pair $\gamma_0$ and $\delta_0$ which intersect in a
single point. By Lemma~\ref{change of coordinates} (3) there exists
$\bar f \in \mcg(S)$ such that $T_{\gamma} = \bar f T_{\gamma_0}
\bar f^{-1}$ and $T_{\delta} = \bar f T_{\delta_0} \bar f^{-1}$. From
this is follows easily that $\phi(T_{\gamma})$ and $\phi(T_{\delta})$
commute.
\end{proof}

\begin{remark} 
Suppose $\gamma$ and $\delta$ are simple closed
curves which intersect transversely in one point. Then
a standard relation in the mapping class group (see Proposition (3.16)
of \cite{FM}) says
\[
T_\gamma T_\delta T_\gamma =  T_\delta T_\gamma T_\delta.
\]
From this it is clear that the hypothesis of part (2) 
of Proposition~(\ref{prop:hom}),
that $\phi(T_\gamma)$ and $\phi(T_\delta)$ commute,
is equivalent to assuming $\phi(T_\gamma) = \phi(T_\delta)$.
\end{remark}

\section{Linear Representations of $\mcg(S)$}

In this section we prove  the first of our main theorems. 

If $\lambda \in \C$ is an eigenvalue of $L\in GL(n,\C)$ then
\[
\ker(L -\lambda I) \subset \ker(L -\lambda I)^2 \subset \dots
\ker(L -\lambda I)^k
\]
is an increasing sequence of   subspaces beginining with the {\em  eigenspace} $\ker(L -\lambda I)$ and ending
with the {\em generalized eigenspace} $ \ker(L -\lambda I)^k$
for $k$ large.   If $L'$ commutes with $L$ then $L'$ preserves each $\ker(L -\lambda I)^j$.

\bigskip

\noindent
{\bf Theorem~\ref{thm:lin-rep}.}
{\em   Suppose that 
 $S$ is a genus $g\ge 1$ surface of finite type (perhaps with boundary
and punctures), that $n < 2g$  and that  $\phi : \mcg(S) \to GL(n,\C)$ is a
homomorphism.   Then the image $H$ of $\phi$ is abelian.  If $g = 2$ then $H$ is finite cyclic; if $g \ge 3$ then $H$ is trivial.    }

\bigskip

\begin{proof}   It suffices to show that $H$ is abelian because in that case  $\phi$
factors through the abelianization of $\mcg(S)$ which is  trivial if $g \ge 3$ and $\Z/10\Z$ if $g = 2$
by  Proposition~(\ref{prop:homology}).

  The proof that $H$ is abelian is by double induction, first on $g$ and then on $n$.     Since $GL(1,\C)$ is abelian the theorem is obvious if $n=1$ and so also if $g=1$.    It therefore suffices to prove the theorem assuming that $g \ge 2$ and that the following inductive hypothesis holds :   If $N$  is a genus $g'\ge 1$ surface of finite type (perhaps with boundary and punctures) and if  either $g'<g$  and  $m <2g'$ or if $g' = g$ and $m < n<2g$ then the image of any homomorphism $\mcg(N) \to GL(m,\C)$ is abelian.

Let $\alpha$ and $\beta$ be non-separating simple closed curves in $S$
that intersect transversely in one point, let $M = M(\alpha,\beta)$ be
a genus one subsurface which is a regular neighborhood of $\alpha \cup
\beta$ and let $S' = S'(\alpha,\beta)$ 
be the genus $g_0 = g-1$ subsurface which is the
closure of $S \setminus M.$ We consider $ \mcg(S')$
to be a subgroup of $\mcg(S)$ via the natural embedding.
Note that $\mcg(S')$ lies in  the
intersection of the centralizers of $T_\alpha$ and $T_\beta$.    Letting $L_\alpha = \phi(T_\alpha)$, $L_\beta = \phi(T_\beta)$ and $H' =  \phi(\mcg(S'))$,  every element of $H'$ commutes with  both $L_\alpha$ and $L_\beta$.


\begin{sublemma}  \label{invariant subspaces} Assume notation  and inductive hypothesis as above.  If at least one of the following conditions is satisfied then $H$ is abelian.
\begin{enumerate}
\item  $\C^n = V_1 \oplus V_2$ where $V_i$ is  an $H'$-invariant
subspace with dimension at most $n-2 $.
\item   $L_\alpha$ has a unique eigenvalue and   there exists an $H'$-invariant subspace $V$  of
$\C^n$ with dimension $2 \le d \le n-2$. 
\item  $L_\alpha$ has a unique eigenvalue and  there is an $H$-invariant subspace  $V$ of $\C^n$ with dimension  $1 \le d \le n-1$. 
\end{enumerate}
\end{sublemma}

\proof   
It suffices to show that $H'$ is abelian because
then there exist simple closed curves $\gamma, \delta$ in 
$S'$ with  $\phi(T_\gamma)$ and $\phi(T_\delta)$ commuting and  $H$ is abelian by Proposition~(\ref{prop:hom})(2).

For case (1), identify $\Aut(V_i)$ with $GL(d_i,\C)$ where
$d_i$ is the dimension of $V_i$ and let $\phi_i : \mcg(S')
\to GL(d_i,\C)$ be the homomorphism induced by $\phi$
restricted to $V_i$.  Since $n-2 < 2g -2 = 2g_0$, the
inductive hypothesis implies that the image of $\phi_i$,
and hence $H'$, is abelian.

For case (2) let $\lambda$ be the  unique eigenvalue
$L_\alpha$ and note that $g_0 \ge 2$ because $2g_0 = 2g-2 > n-2 \ge 2$.  Identify $\Aut(V)$ with
$GL(d,\C)$ and let $\phi_1 : \mcg(S') \to GL(d,\C)$ be the
homomorphism induced by $\phi$. Similarly, identify $\Aut(\C^n/V)$
with $GL(n-d,\C)$ and let $\phi_2 : \mcg(S') \to GL(n-d,\C)$ be the
homomorphism induced by $\phi$.  Since $d,n-d \le n-2 = 2g_0 $, the
inductive hypothesis implies that $\phi_1(\eta)$ and $\phi_2(\eta)$
have finite order, and hence (considering the Jordan canonical form)
are diagonalizable, for all $\eta \in \mcg(S')$.  If $\phi(\eta)$ is
conjugate to $L_\alpha$ then $\lambda$ is its unique eigenvalue so
$\phi_d(\eta) = \lambda I_d$ and $\phi_{n-d}(\eta) = \lambda I_{n-d}$.
In other words, there is a basis for $\C^n$, with respect to which any
such $\phi(\eta)$ is represented by a matrix that differs from
$\lambda I$ only on the on upper right $d \times (n-d)$ block.  It is
easy to check that any two such matrices commute.    Since $\mcg(S')$
has a generating set consisting of elements that are conjugate to $T_\alpha$, $H'$ has a generating set consisting of elements that are conjugate to $L_\alpha$ and $H'$ is abelian.

The argument for case (3) is the same as that for case (2)
except that it is applied directly to $H$ and not $H'$.  We
can apply this to dimension $1$ and $n-1$ because $S$ has genus
$g$ and not $g-1$.
\endproof 

WIth Sublemma~\ref{invariant subspaces} in hand we prove Theorem~\ref{thm:lin-rep} by a case analysis.

\bigskip

\noindent{\bf Case 1:  ($L_\alpha$ has three distinct eigenvalues or has two distinct eigenvalues whose generalized eigenspaces do not have dimension $1$ or $n-1$.)}     \ \ \    Since each generalized eigenspace of $L_\alpha$ is $H'$-invariant and since $\C^n$ is the product of the generalized eigenspaces of $L_\alpha$, this case follows from Subemma~\ref{invariant subspaces}(1).

\bigskip
 
\noindent{\bf Case 2:  ($L_\alpha$ has a unique eigenvalue $\lambda$)}   \ \ \   The proof in this case requires a subcase analysis depending on the dimension $r$ of the eigenspace $W_\alpha$.  We will make use of the fact that $L_\beta$ is conjugate to $L_\alpha$ (and so has   unique eigenvalue $\lambda$ and   eigenspace $W_\beta$) and that both  $W_\alpha$ and $W_\beta$ are $H'$-invariant.

\bigskip

\noindent{\bf Case 2A:   ($r \ne 1, n-1$)}  \ \ \    If $r =n$ then   $L_\alpha = \lambda I$.  Since $H$ is generated by conjugates of $L_\alpha$, each of which is multiplication by a scalar, it is abelian.  If  $2 \le r \le n-2$ then  Sublemma~\ref{invariant subspaces}(2)  applied to $V= W_\alpha$ implies that $H$ is abelian.


\bigskip

 \noindent{\bf Case 2B:  (
 $W_\alpha= W_\beta$)}  \ \ \ 
  For any pair $\alpha', \beta'$ of non-separating 
simple closed curves intersecting in a point, 
$W_{\alpha'} = W_{\beta'}$ by Lemma~(\ref{change of coordinates}) (3).  Corollary~\ref{cor:pairs in P} implies that if a non-separating simple
closed curve $\gamma$ is disjoint from $\alpha$ then there is a
non-separating simple closed curve $\delta$ that intersects both
$\alpha$ and $\gamma$ in a single point.  Applying the above result
with the pair $(\alpha, \beta)$ replaced first by $(\alpha, \delta)$
and then by $(\delta, \gamma)$ we conclude that $W_\alpha = W_\delta =
W_\gamma$.  It follows that $V = W_\alpha$ is an eigenspace for a
generating set of $H $ and in particular is $H$-invariant.  Sublemma~\ref{invariant subspaces}(3) completes the proof.

\bigskip

\noindent{\bf Case 2C:  ($r =1$ and $W_\alpha\ne W_\beta$)}  \ \ \  
    If $n \ge 4$, then  Sublemma~\ref{invariant subspaces}(2) applied to $V= W_\alpha \oplus W_\beta$ completes the proof.    If $n=2$, then  $\C^2 = W_\alpha \times W_\beta$.  Since $W_\alpha$ and $W_\beta$ are $H'$-invariant, there is a basis of $\C^2$ with respect to which each element of $H'$ is diagonal so $H'$ is abelian.   
    
    Suppose then that $n=3$ and hence that  there is a basis for $\C^3$ with respect to which $L_\alpha$ has matrix $\lambda I +N$ where
$N$ is the $3 \times 3$ matrix equal to $E_{12}+ E_{23}$ where
$E_{ij}$ is the matrix all of whose entries are $0$ except the
$ij^{th}$ which is $1.$
The subspaces $\ker(L -\lambda I)
\subset \ker(L -\lambda I)^2 \subset \ker(L -\lambda I)^3 = \C^3$
are all distinct and $H'$-invariant.   
The matrices for elements of $H'$ are therefore upper
triangular in the same basis for which the matrix of $L$
is $\lambda I +N$.  There are generators of $H'$ that are conjugate to $L_\alpha$ and so have unique eigenvalue $\lambda$.   The diagonal entries for the  matrices for these generators 
are all $\lambda.$  A straightforward computation using the
fact that these generators commute with $L_\alpha$ shows that
their matrices have the form $\lambda I + aN + bN^2$ but the
group generated by matrices of this form is abelian since all
elements are polynomials in $N$.  This proves that $H'$ is abelian.

\bigskip

\noindent{\bf Case 2D:  ($r =n-1$ and $W_\alpha\ne W_\beta$)}  \ \ \       If $n \ge 4$, then  Sublemma~\ref{invariant subspaces}(2) applied to $V= W_\alpha \cap W_\beta$ completes the proof.    If $n =2$ then $r=1$ and case 2C applies.    

Suppose then that $n=3$ and $r=2$.

The image $U_\alpha = (L_\alpha -\lambda I)(\C^3)$ of
$L_\alpha - \lambda I$ is contained in $W_\alpha =
\ker(L_\alpha -\lambda I)$ because $(L_\alpha -\lambda I)^2
= 0$; since $W_\alpha$ has dimension $2$, $U_\alpha$ has
dimension $1$.  Moreover, if $W$ is any two-dimensional
$L_\alpha$-invariant subspace of $\C^3$ then $(L_\alpha
-\lambda I)(W) \subset W$ and also $(L_\alpha -\lambda I)(W)
\subset U_\alpha.$ Indeed $(L_\alpha -\lambda I)(W) =
U_\alpha$ unless $W = W_\alpha.$ Hence $U_\alpha$ is a
one-dimensional subspace of every $L_\alpha$-invariant
two-dimensional subspace of $\C^3.$ We conclude that if $W_1
\ne W_2$ are two-dimensional $L_\alpha$-invariant subspaces
then $U_\alpha = W_1 \cap W_2.$
  
Suppose that $\alpha'$ is a simple closed non-separating
curve in $S$.  If $\alpha'$ is disjoint from $\alpha$ then
$L_{\alpha'} = \phi(T_{\alpha'})$ commutes with $L_\alpha$
and so preserves $U_\alpha$.  We claim that the same is true
if $\alpha'$ intersects $\alpha$ in a single point.  By
Lemma~\ref{change of coordinates}(3) we may assume that
$\alpha' = \beta$.  Choose simple closed non-separating
curves $\gamma, \delta \subset S'$ that intersect in a
single point and let $W_\gamma$ and $W_\delta$ be the unique
eigenspaces for $L_\delta = \phi(T_\delta)$ and $L_\gamma =
\phi(T_\delta)$.  Both $W_\gamma$ and $W_\delta$ are
$L_\alpha$-invariant.  Lemma~(\ref{change of coordinates})
(3) implies that $W_\gamma \ne W_\delta$ and hence that
$U_\alpha = W_\gamma \cap W_\delta$.  For the same reason,
$U_\beta= W_\gamma \cap W_\delta$.  This proves that $
U_\beta=U_\alpha$ and hence that $U_\beta$ is
$L_\alpha$-invariant as claimed.
        
        We have now shown that there is a generating set for $H$ that preserve $U_\alpha$ so Sublemma~\ref{invariant subspaces}(3) completes the proof.
        
\bigskip

\noindent{\bf Case 3  ($L_\alpha$ has two eigenvalues; the dimensions of the generalized eigenspaces are $1$ and $n-1$)}   \ \ \

Let $U_\alpha$ and $W_\alpha$ be the dimension $1$ and dimension $n-1$  generalized eigenspaces of $L_\alpha$ respectively; define   $U_\beta$ and $W_\beta$ similarly.  All
of these spaces are $H'$-invariant.  The argument given for case 2B applies in this context  if either
$U_\alpha = U_\beta$ or if $W_\alpha = W_\beta$.  We may therefore assume that  $U_\alpha \ne U_\beta$ and $W_\alpha \ne W_\beta$.   Corollary~\ref{cor:pairs in P}  and Lemma~\ref{change of coordinates} (2) imply that if   $\gamma$ is a  non-separating simple closed curve that is disjoint from $\alpha$ and whose union with $\alpha$    does not separate $S$ then      $U_\alpha \ne U_\gamma$  and $W_\alpha \ne W_\gamma$ (where $U_\gamma$ and $W_\gamma$ are defined in the obvious way.)

\bigskip
\noindent{\bf Case 3A ($n=2$ or $3$)} \ \ \ 
If $n=2$, then   $U_\alpha$ and $W_\alpha$ are one dimensional and $H'$-invariant so there is a basis with respect to which the matrix for  each element of $H'$ is diagonal.  Thus  $H'$ is abelian  and we are done.  
 
If $n = 3$, choose simple closed non-separating curves
$\gamma,\delta \subset S'$ that intersect in a point.  Since
$U_\alpha$ is a one dimensional $L_\gamma$-invariant
subspace that is not equal to $U_\gamma$ it must be
contained in $W_\gamma$.  The same argument implies that
$U_\beta \subset W_\gamma$ and hence that $W_\gamma$ is
spanned by $U_\alpha$ and $U_\beta$.  For the same reason,
$W_\delta$ is spanned by $U_\alpha$ and $U_\beta$.  But then
$W_\gamma = W_\delta$ in contradiction to Lemma~\ref{change
  of coordinates} (3).

\bigskip
   
\noindent{\bf Case 3B ($g > 3, \ n \ge 4$)} \ \ \ 
If $g > 3$ then the induction hypothesis implies that $H'$
acts trivially on $W_\alpha \cap W_\beta$.  Hence if $\gamma$
is a non-separating simple closed curve in $S'$ the eigenvalue
of $L_\gamma$ with multiplicity $(n-1)$ must be $1$ since
the dimension of $W_\alpha \cap W_\beta$ is $> 1.$ 
Therefore the eigenvalue of $L_\alpha$ corresponding to the
eigenspace $W_\alpha$ is $1$
and the determinant of $L_\alpha$ is the eigenvalue
corresponding to $U_\alpha$; in particular the deteminant of
$L_\alpha$ is not $1$.  But postcomposing $\phi$ with the
determinant defines a homomorphism from $\mcg(S)$ to $\C^*$
which must be trivial.  This contradiction shows that this
case can not happen.

\bigskip

\noindent{\bf Case 3C ($g \le 3, n=4$ or $5$)}    \ \ \  We may assume that $g = 3$ because $2g > n$.  Varying our notation slightly,  we choose three pairs of simple
closed curves $(\alpha_i, \beta_i),\ 1\le i\le 3$, such that
$\alpha_i$ and $\beta_i$ intersect transversely in a single point
and $\alpha_i \cup \beta_i$ is disjoint from $\alpha_j \cup \beta_j$ 
for $i \ne j$ and such that no two from among these six curves
separate $S$.

We apply the following observation twice: if $\gamma$ and $\delta$ are disjoint simple closed
curves whose union does not separate then   $U_\gamma \subset
W_\delta$.  This follows from the fact that  $L_\gamma$ and $L_\delta$ commute and   $U_\gamma \ne U_\delta$ is $L_\delta$-invariant  and so is  contained
in an eigenspace of  $L_\delta.$   

By hypothesis,   $U_{\alpha_1}$ and $U_{\beta_1}$ span a 
two-dimensional subspace of $\C^n$.    By the  above observation,
$U_{\alpha_1}, U_{\beta_1} \subset W_{\alpha_2}$ and
hence the span $V_0$ of $U_{\alpha_1}, U_{\beta_1}$ and $U_{\alpha_2}$
is three dimensional.  

Applying the observation above a second time we note that
the span of $U_{\alpha_1}, U_{\beta_1}$ and $U_{\alpha_2}$
is contained in both $W_{\alpha_3}$ and $W_{\beta_3}$  and
hence in their intersection. We conclude that the dimension
of $W_{\alpha_3} \cap W_{\beta_3}$ is three since $n \le 5$ 
and $W_{\alpha_3} \ne W_{\beta_3}$ implies it cannot be greater
than three.  Therefore $V_0 = W_{\alpha_3} \cap W_{\beta_3}$.

It is now straightforward to extend the six curves above to 
a set of non-separating simple closed curves such that Dehn
twists about these curves form a generating set for $\mcg(S)$
and each of these curves is either disjoint from 
$\alpha_1 \cup \beta_1 \cup\alpha_2$ or disjoint
from $\alpha_3 \cup \beta_3$ (see Corollary 4.16 and ff. of \cite{FM}).
It follows from the two descriptions of $V_0$ (as 
$W_{\alpha_3} \cap W_{\beta_3}$ and the span of 
of $U_{\alpha_1}, U_{\beta_1}$ and $U_{\alpha_2}$)
that if $\gamma$ is any one of these generators then
$V_0$ is $L_\gamma$-invariant.  Hence $V_0 \subset \C^5$
is $H$ invariant and Sublemma~\ref{invariant subspaces}(3) completes the proof.  \end{proof}

\section{Homomorphisms to  $\Homeo(\T^2)$.}\label{sec:torus}

\begin{lemma}\label{lem:abelian}
If $G$ is a finite group which acts 
on $\T^2$ by homeomorphisms isotopic to the identity,
then the action factors through an abelian group which acts
freely on $\T^2$. 
\end{lemma}

\begin{proof}
Suppose $\phi: G \to \Homeo(\T^2)$ is a
homomorphism.
If $g \in G$ and 
$\phi(g) \in \Homeo(\T^2)$ has a fixed point then
$\phi(g)=id,$ since every finite order homeomorphism of the
torus which is isotopic to the identity and has a fixed
point must be the identity (see \cite{Hartzman} or \cite{Brouwer}).
It follows that if $A = G/ker(\phi)$
then $A$ acts freely on $\T^2$. 

Let $\bar \phi: A \to \Homeo(\T^2)$ be the injective homomorphism
induced by $\phi$ and let $M$ be the quotient of $\T^2$ by this
action.  Since $A$ acts freely, $M$ is a closed surface, and the
projection $p: \T^2 \to M$ is a $k$-fold covering for some $k$. The
Euler characteristic of $M$ is $1/k$ times the Euler characteristic of
$\T^2$ and so is zero, proving that $M$ is homeomorphic to a torus and
hence has abelian fundamental group.
Since $A$ is isomorphic to the group of covering
translations for $p: T^2 \to M$, $A$ is isomorphic to a quotient of
$\pi_1(M)$ and hence is abelian.
\end{proof}

We can now provide the proof of our result about homomorphisms
to $\Homeo(\T^2).$
\medskip

\noindent
{\bf Theorem~\ref{thm:torus}.}
{\em Suppose that S is a closed oriented surface with 
genus $g \ge 3$. Then every homomorphism  $\phi: \mcg(S) \to \Homeo(\T^2)$
is trivial.}

\begin{proof}
We first observe that $\mcg(\T^2) \cong SL(2, \Z)$ so
by Theorem~(\ref{thm:lin-rep})  the map from 
$\mcg(S)$ to $\mcg(\T^2)$ induced by $\phi$ is trivial.
Hence the image of $\phi$ lies in  
$\Homeo_0(\T^2)$, the homeomorphisms isotopic to the identity.

Let $H$ be the subgroup of $SO(3)$ generated
by $u$, a rotation of $2\pi/g$ about the $z$-axis,
and $v$, a rotation of $\pi$ about the $x$-axis.
Then $H$ is a group of order $2g$ which is non-abelian (this
uses $g \ge 3$). In fact $H$ is a semi-direct product of 
$\Z/g\Z$ with $\Z/2\Z$, with the non-trivial element of $\Z/2\Z$ 
acting on $Z/g\Z$ by sending an element to its inverse.
One calculates easily
\[
v^{-1}u^{-1}vu = u^2.
\]

We want to show that if
$S$ is a closed surface of genus $g \ge 3$ then
there is a subgroup of $\mcg(S)$ isomorphic to
$H$. Let $M_0$ be the unit sphere in $\R^3$ and
note that it is $H$ invariant.
The $H$ orbit of the point $(1,0,0) \in S^2$
consists of $g$ points uniformly spaced about the equator.
We remove a small open disk about each of these points thus
creating a genus zero surface $M$ with $g$ boundary components which
is $H$ invariant.
Next we attach $g$ genus one surfaces each with one boundary component to
each of the boundary components of $M$.  This creates a surface
$S$ of genus $g$.  If we make these attachments in a symmetric
fashion we may assume that $S$ is invariant under the group
$H$ acting on $\R^3.$ The restriction to $S$ of 
each non-trivial element of $H$ represents
a non-trivial element of $\mcg(S)$ and the assignment of an
element of $H$ to this representative in $\mcg(S)$ is an 
embedding of $H$ in $\mcg(S).$

The restrictions of $u$ and $v$ to $S$ are representatives of elements
$U, V \in \mcg(S)$ and $U$ and $V$ generate a subgroup of $\mcg(S)$
isomorphic to $H$.  Suppose $\phi: \mcg(S) \to \Homeo(\T^2)$ is a
homomorphism.  Since $U^2 = V^{-1}U^{-1}VU$ is a commutator,
Lemma~(\ref{lem:abelian}) implies that $\phi(U^2) = id \in
\Homeo(\T^2).$ On the other hand $u^2$ induces a non-trivial
permutation of the $g$ tori we attached to $M$.  Let $\alpha$ be a
non-separating simple closed curve in one of these tori and let $\beta
= u^2(\alpha).$ Then $\alpha$ and $\beta$ are disjoint simple closed
curves in $S$ whose union does not separate.  Since
\[
T_\beta = U^2 T_\alpha U^{-2},
\]
we have 
\[
\phi(T_\beta) = \phi(U^2) \phi( T_\alpha) \phi(U^{-2}) = \phi( T_\alpha).
\]
Part (1) of Proposition~(\ref{prop:hom}) now implies that $\phi$ is trivial.
\end{proof}

\section{Homomorphisms to  $\Diff(S^2)$.}

\noindent
{\bf Theorem~\ref{thm:S2}.}
{\em Suppose that S is a closed surface 
of genus $g \ge 7$.
Then every homomorphism $\phi: \mcg(S) \to \Diff(S^2)$
is trivial.}

\begin{proof}
If the genus $g$ is odd we let $S_0$ be a torus and
if $g$ is even we let it be a surface of genus two.
In either case we embed $S_0$ in $\R^3$ so that it 
is invariant under rotation by $\pi$
about the $z$-axis and let $R$ denote this  rotation.
We consider
two surfaces $Y^\pm$ of genus $(g-1)/2$ or genus
$(g-2)/2$ (depending on the parity of $g$)  each with a single boundary
component. We remove two open disks from $S_0$, which
are symmetrically placed so that they are interchanged by $R$,
and attach the surfaces $Y^\pm$ along boundaries to the
boundary components of $S_0$ thus created.  We do this in
such a way that $Y^+$ lies in the half space with $x$-coordinate
positive and $T(Y^\pm) = Y^\mp.$ Then $S = S_0 \cup Y^+ \cup Y^-$
has genus $g$ and $R: S \to S,$ the restriction of $R$ to $S$,
is an involution which induces
an element of order two $\tau \in \mcg(S).$  
The element $\tau$ induces
an inner automorphism which we denote by $\tau_\#: \mcg(S) \to \mcg(S).$

There are embeddings $\mcg(Y^\pm)\to \mcg(S)$ induced
by the embeddings of $Y^\pm$ in $S$.  Let $G^\pm$ denote the subgroup
of $\mcg(S)$ which is the image of the embedding of
$\mcg(Y^\pm).$ Note that elements of $G^+$ commute
with elements of $G^-$ since they have representatives with disjoint
support.  Also note that $\tau_\#(G^\pm) = G^\mp.$

Let $G_0$ be the subgroup of $\mcg(S)$ consisting of all elements
of the form $\eta \tau_\#(\eta)$ for $\eta \in G^+$.  It is clear
that the correspondence $\eta \mapsto \eta \tau_\#(\eta)$ defines
an isomorphism from $G^+$ to $G_0$ and hence $G_0$ is isomorphic
to $\mcg(Y^+).$ Observe that
\[
\tau_\#( \eta \tau_\#(\eta)) = \tau_\#(\eta) \eta = \eta \tau_\#(\eta),
\]
so $\tau$ commutes with every element of $G_0.$ 

Now let $\alpha$ be a non-separating simple closed in 
$Y^+$ and let $\beta = R (\alpha).$
Recall that $T_\alpha$ denotes the left Dehn twist about $\alpha$.

We want to
show that $\phi(T_\alpha) = \phi(T_\beta)$ or $\phi(T_{\beta})^{-1}$.
Since $\alpha$ and $\beta$ are disjoint simple closed curves in
$S$ whose union does not separate it will follow from part (1) of
Proposition~(\ref{prop:hom}) that $\phi: \mcg(S) \to \Diff(S^2)$
is trivial.

To do this we consider $h = \phi(\tau) \in \Diff(S^2).$
The diffeomorphism $h$ of $S^2$ has order at most two.
If $h$ is  is the identity then
\[
\phi(T_\alpha) =  h \circ \phi(T_\alpha) \circ h^{-1} =
\phi(\tau T_\alpha \tau^{-1}) = \phi(T_\beta).
\]
Hence we may assume $h$ has order two.  In that
case $h$ has exactly two fixed points.  To see this
note that the derivative of $h$ at a fixed point can
only be $-I$ (since $h$ preserves orientation) and
hence the index of each fixed point is $+1$.  Since
the sum of the fixed point indices is $2$, the Euler
characteristic of $S^2$, there must be two fixed points.
The set of  fixed points must be preserved by each element 
of $\phi(G_0)$ since these elements commute with $h$. But also if some 
elements of $\phi(G_0)$ permute these fixed points that
defines a non-trivial homomorphism from $G_0$ to $\Z/2\Z$.
Such a homomorphism cannot exist since
of $G_0 \cong \mcg(Y^+)$ is perfect by Proposition~(\ref{prop:homology}).  
We conclude there is a global fixed point $p \in S^2$ for $G_0$.

By Theorem~(\ref{thm:lin-rep}) any homomorphism from 
$G_0 \cong \mcg(Y^+)$ to $GL(2, \R)$ is trivial.  
We can conclude 
that the derivative $D_p(\phi(\eta)) = I$ for every $\eta \in G_0.$

Since $G_0$ is perfect $\phi(G_0)$ admits no
non-trivial homomorphisms to $\R$.
Applying Thurston's stability theorem \cite{Thurston:stability}
to $\phi(G_0)$ we conclude that $\phi(G_0)$ is trivial.
Finally we observe that if we let
\[
\eta = T_\alpha T_\beta = T_\alpha \tau_\#(T_\alpha) \in G_0,
\]
then $\phi(\eta) = id$ implies
\[
\phi(T_\alpha) = \phi(T_\beta)^{-1}.
\]
Once again Proposition~(\ref{prop:hom}) (1) implies 
that $\phi: \mcg(S) \to \Diff(S^2)$
is trivial.
\end{proof}

{\bf Theorem~\ref{thm:punc}.}
{\em Suppose that S is an oriented surface with 
genus $g$ with $k$ punctures and no boundary.  
If $g \ge 7$ and $0 \le k \le 2g-10,$
then every homomorphism $\phi: \mcg(S) \to \Diff(S^2)$
is trivial.}

\begin{proof}
If $S_0'$ is a surface of genus $g_0 \ge 1$ then it can be embedded in
$\R^3$ in such a way that it is invariant under $R$, rotation
about the $z$-axis by $\pi,$ and so that it intersects the
$z$-axis in $2g_0 + 2$ points.  So $R$ has $2g_0 + 2$ fixed points.

If $k \le 2g_0 +2$ we let $S_0$ be $S_0'$ punctured at
$k$ of the points of $\Fix(R)$ so $R$ represents an element
of $\mcg(S_0)$.  We now repeat the proof of Theorem~\ref{thm:S2}
using this $S_0$ instead of the surface of genus one or two labeled
$S_0$ in that proof.

The surface $S$ is constructed as before by attaching two
surfaces of genus $\ge 3$ to $S_0$.  Hence the resulting
surface $S$ has genus $g \ge g_0 + 6 \ge 7.$ We then
conclude as before that $\phi$ is trivial. The restriction on $k$ is
\[
k \le 2g_0 + 2 \le 2(g-6) + 2 = 2g -10.
\]
\end{proof}

\bibliographystyle{plain}
\bibliography{mcg}
\end{document}